\documentclass[12pt,reqno]{article}

\usepackage{amssymb}
\usepackage{amsmath}
\usepackage{amsthm}
\usepackage{amsfonts}
\usepackage{amscd}
\usepackage{graphicx}

\usepackage{hyperref}
\usepackage{url}

\usepackage{fullpage}
\usepackage{float}
\usepackage[usenames,dvipsnames]{xcolor}

\usepackage{graphics}
\usepackage{latexsym}
\usepackage{epsf}
\usepackage{tikz}
\usepackage[labelfont=rm]{subcaption}
\usetikzlibrary{automata,positioning}
\tikzset{every state/.style={minimum size=0pt}}

\def\andd{\wedge}
\def\Zee{\mathbb{Z}}
\def\suchthat{\, : \,}

\newcommand{\seqnum}[1]{\href{https://oeis.org/#1}{\rm \underline{#1}}}

\DeclareMathOperator{\ce}{ce}

\begin{document}

\author{Nicolas Ollinger\\
Universit\'e d'Orl\'eans\\
LIFO - b\^at.\ 3IA\\
6 rue L\'eonard de Vinci\\
BP 6759\\
45067 Orl\'eans Cedex 2\\
France\\
\href{mailto: nicolas.ollinger@univ-orleans.fr}{\tt nicolas.ollinger@univ-orleans.fr}\\
\and Jeffrey Shallit\footnote{Supported in part by a grant from NSERC, 2018--04118.}\\
School of Computer Science\\
University of Waterloo\\
200 University Ave. W.\\
Waterloo, ON  N2L 3G1\\
Canada\\
\href{mailto:shallit@uwaterloo.ca}{\tt shallit@uwaterloo.ca}}

\title{The Repetition Threshold for Rote Sequences}

\theoremstyle{plain}
\newtheorem{theorem}{Theorem}
\newtheorem{corollary}[theorem]{Corollary}
\newtheorem{lemma}[theorem]{Lemma}
\newtheorem{proposition}[theorem]{Proposition}

\theoremstyle{definition}
\newtheorem{definition}[theorem]{Definition}
\newtheorem{example}[theorem]{Example}
\newtheorem{conjecture}[theorem]{Conjecture}
\newtheorem{problem}[theorem]{Problem}

\theoremstyle{remark}
\newtheorem{remark}[theorem]{Remark}

\maketitle

\begin{abstract}
We consider Rote words, which are infinite binary words with factor complexity $2n$.  We prove that the repetition threshold for this class is $5/2$.
Our technique is purely computational, using the {\tt Walnut} theorem prover and a new technique for generating automata from morphisms due to the first author and his co-authors.
\end{abstract}

\section{Introduction}

We consider infinite sequences (words) $\bf x$ over a finite alphabet.   A classical metric for such sequences is their {\it factor complexity} (aka subword complexity, aka complexity) $\rho_{\bf x} (n)$, which counts the number of distinct factors (i.e., contiguous blocks) occurring in $\bf x$.

Rote sequences, introduced by Rote \cite{Rote:1994}, form a particularly interesting class; they are the words of factor complexity
$2n$ for $n \geq 1$.   Since such a sequence has exactly two factors of length $1$, Rote sequences
must be defined over a binary alphabet, which we take to be $\Sigma_2 = \{0,1\}$.
Some examples of Rote sequences studied in the literature include the
{\it Rote-Fibonacci} sequence \cite{Du&Mousavi&Rowland&Schaeffer&Shallit:2017} and
complementation-symmetric Rote sequences \cite{Dvorakova&Medkova&Pelantova:2020}.

Another metric of interest involves the degree of repetition in a word.   We say a finite
word $z = z[0..n-1]$ has a {\it period\/} $p\leq n$
if $z[i]=z[i+p]$ for $0 \leq i < n-p$.
For example, the French word {\tt entente} has
periods $3$, $6$, $7$.   The {\it exponent\/} $\exp(z)$ of
a finite word $z$ is its length divided by its
shortest period.   Thus, {\tt entente} has
exponent $7/3$.  The {\it critical exponent\/} $\ce({\bf x})$ of an infinite word $\bf x$ is the supremum, over all finite factors $z$, of  $\exp(z)$.  Such a critical exponent can be attained, or not attained.  Given a set $S$ of infinite words, its
{\it repetition threshold\/} is the infimum, over
all ${\bf x} \in S$, of $\ce({\bf x})$.  

Repetition thresholds have been determined for many large classes of words.   For example, in Dejean's conjecture \cite{Dejean:1972}, the set
$S$ is taken to be $\Sigma_k^*$, where
$\Sigma_k = \{ 0,1,\ldots, k-1 \}$.
Dejean's conjecture was finally proven
by Currie and Rampersad \cite{Currie&Rampersad:2011} and
\cite{Rao:2011}, independently.   Other interesting classes that have been studied include Sturmian words \cite{Carpi&deLuca:2000}, 
balanced words 
\cite{Rampersad&Shallit&Vandomme:2019}, and many others.

In this note we determine the repetition threshold for Rote sequences.  Along the way we construct a certain infinite binary word $\bf q$ and examine some of its properties.

\section{Lower bound}

Let us extend the notion of Rote word to finite words:  we say a finite word $z$ of length $n$ is Rote
if $\rho_z(i) \leq 2i$ for $1 \leq i \leq n$.

\begin{proposition}
Every Rote word of length $> 38$ (and hence every infinite Rote word) has a finite factor of exponent $\geq 5/2$.
\end{proposition}

\begin{proof}
We perform breadth-first or depth-first search over the tree of all finite Rote words of critical exponent $< 5/2$.   This tree is finite and the longest words are of length $38$.  They are
$$00110011010011001001101001100100110010,
00110011010011001001101001100100110011,$$
and the six words formed by reversal and
complementation of these two.
\end{proof}

\section{The words $\bf p$ and $\bf q$}

It now remains to see that the bound of $5/2$ is optimal; that is, there exists an infinite binary word of critical exponent $5/2$ and factor complexity $2n$.

We describe this word as follows:  let
$h$ be the morphism defined by
$h(0) = 01$, $h(1) = 21$, $h(2) = 0$.  This morphism and its infinite fixed point
$${\bf p} = 012102101021012101021 \cdots$$ were studied, for example, in
\cite{Currie&Ochem&Rampersad&Shallit:2022}.
Next, define another morphism as follows:
$g(0) = 011$, $g(1) = 0$, and $g(2) = 01$.
The word we are interested in is
${\bf q} := g({\bf p})$.   
Our main result is therefore the following:
\begin{theorem}
  The infinite word ${\bf q}$
is a Rote word with the optimal critical exponent $5/2$.
\label{thm2}
\end{theorem}

The presence of the factor $1001\,1001\,10$
shows that $\ce({\bf q}) \geq 5/2$.
Thus, it remains to prove the following two claims:
\begin{itemize}
\item[(1)] $\bf q$ contains no factors of exponent $>5/2$:  it is $(5/2)^+$-power-free;
\item[(2)] $\bf q$ has factor complexity $2n$.    
\end{itemize}

To prove both these claims, we use the the theorem prover {\tt Walnut}, \cite{Mousavi:2016,Shallit:2023}, a free software package that can rigorously prove or disprove many assertions about certain kinds of sequences.   It is discussed in the next section.

The word $\bf q$ was very recently studied by Dvo\v{r}\'akov\'a, 
Ochem, and Opo\v{c}ensk{\'a} \cite{Dvorakova&Ochem&Opocenska:2024};
they also proved, as we have,
that $\bf q$ is $(5/2)^+$-power-free,
using a detailed analysis of return words.   The reader may
enjoy comparing their proof with ours.  We thank Pascal Ochem for informing us about their work.

\section{Proving theorems with {\tt Walnut}}
Both claims (1) and (2) above are, in principle, decidable because we know that
$\bf p$ is an automatic word in the Pisot-$4$ numeration system described in
\cite{Currie&Ochem&Rampersad&Shallit:2022}.
In particular, there is a $72$-state automaton computing ${\bf p}[n]$ in the following sense: the input is $n$ expressed in the $P4$ numeration system, and the output associated with the last state reached is 
${\bf p}[n]$.   Therefore, by a classic result of Bruy\`ere et al.~\cite{Bruyere&Hansel&Michaux&Villemaire:1994}, we know that there is an algorithm to translate a first-order logical formula $\varphi$, using addition and indexing into ${\bf p}$, into an automaton accepting the representation of the natural number values making $\varphi$ true.   If there are no free variables, this gives us a way to rigorously prove or disprove claims merely by computing the corresponding automata.  

{\tt Walnut}, a free software system originally created by Hamoon Mousavi, \cite{Mousavi:2016}, and discussed in more detail in \cite{Shallit:2023}, implements this algorithm.
As an example, let us find a DFAO (deterministic finite automaton with output) computing $\bf q$.
We can get it from the DFAO for $\bf p$ because, as proved in \cite{Currie&Ochem&Rampersad&Shallit:2022},
the count of each letter appearing in prefixes of $\bf p$ is also computable by
an automaton, in a synchronized sense \cite{Shallit:2021h}:
there are automata $A_i$, for $i \in \{0,1,2\}$
that on input $n$ and $y$ in parallel, both
expressed in the $P4$ system, and accepts
if and only if $y = |{\bf p}[0..n-1]|_i$, the number of occurrences of $i$ in the prefix
of length $i$.  Here the automaton   $A_0$ has $43$ states, $A_1$ has $34$ states, and $A_2$ has $48$ states.

Thus we can compute ${\bf q}[n]$ with an automaton by determining the position
of $\bf p$ that gives rise to the symbol, using
these synchronized automata.   This gives an automaton of $72$ states for $\bf q$ in the $P4$ system, which can be computed with the following {\tt Walnut} code.
\begin{verbatim}
def imlen "?msd_pisot4 Ex,y,z $paut0(n,x) & 
   $paut1(n,y) & $paut2(n,z) & w=3*x+y+2*z":
#159 states
def se "?msd_pisot4 ~En $imlen(n,w)":
#72 states
combine Q se:
\end{verbatim}

As an application, let us compute
the abelian complexity of $\bf q$.   (Recall that the abelian complexity of a sequence is the number of distinct blocks of length $n$, up to rearrangement of symbols.)
Since $\bf q$ is defined over a binary alphabet, this amounts to determining how many different possibilities there are for the number of $0$'s in a length-$n$ factor of $\bf q$.   And this, in turn, 

Next, we count the number of $0$'s in a prefix of
$\bf q$ of length $n$, and an arbitrary factor of
length $n$ as follows:
\begin{verbatim}
def pref0 "?msd_pisot4 (n=0&p=0) | 
   (Ew,x $imlen(p-1,w) & $imlen(p,x) & n>w & n<=x)"::
# 182 states
def count0 "?msd_pisot4 Ew,x $pref0(i,w) & $pref0(i+n,x) & z+w=x":
#20474 states  
\end{verbatim}

Next, we determine an automaton that accepts $n$ and $x$ in parallel, if $x$ is the number of $0$'s in some factor of $\bf q$ of length $n$:
\begin{verbatim}
def range0 "?msd_pisot4 Ei $count0(i,n,x)":
# 728 states  
\end{verbatim}
We are now ready to prove the following result.
\begin{theorem}
The abelian complexity function of $\bf q$ takes the values $\{1,2,3,4\}$ and these are the only possibilities.  Furthermore, there is a DFAO of $341$ states computing it.
\end{theorem}

\begin{proof}
First let us show the abelian complexity never takes the value $5$:
\begin{verbatim}
eval val5 "?msd_pisot4 ~Ex,y,n $range0(n,x) & $range0(n,y) & y=x+4":
\end{verbatim}
and {\tt Walnut} returns {\tt TRUE}.

Next, let us determine the automaton computing the abelian complexity function of $\bf q$:
\begin{verbatim}
def val1 "?msd_pisot4 Ax,y ($range0(n,x) & $range0(n,y)) => x=y"::
def val2 "?msd_pisot4 Ax,y ($range0(n,x) & $range0(n,y) & x>=y) => 
   (x=y|x=y+1)"::
def val3 "?msd_pisot4 Ax,y ($range0(n,x) & $range0(n,y) & x>=y) => 
   (x=y|x=y+1|x=y+2)"::
def val4 "?msd_pisot4 Ax,y ($range0(n,x) & $range0(n,y) & x>=y) => 
   (x=y|x=y+1|x=y+2|x=y+3)"::
def abel2 "?msd_pisot4 $val2(n) & ~$val1(n)"::
# 336 states
def abel3 "?msd_pisot4 $val3(n) & ~$val2(n)"::
# 340 states
def abel4 "?msd_pisot4 $val4(n) & ~$val3(n)"::
# 333 states
combine BEL val1=1 abel2=2 abel3=3 abel4=4::
# 341 states
\end{verbatim}
The automaton {\tt BEL} is the desired one computing the abelian complexity function of $\bf q$, and it has $341$ states, too large to display here.
\end{proof}

\section{The main result}

Once we have an automaton that computes
${\bf q}$, we can computationally prove conditions (1) and (2) as follows, at least in principle.    Let's start with condition (1).  It suffices to give a first-order logical statement asserting the nonexistence of factors of exponent $>5/2$.   We can do this as follows:
$$
\neg \exists i, n \ n\geq 1 \andd \forall t, u \ (t\geq i \andd t \leq i + 3n/2 \andd u=t+n) \implies {\bf q}[u]={\bf q}[v].$$
This statement, suitably translated into {\tt Walnut}, looks like
\begin{verbatim}
eval check52plus "?msd_pisot4 ~Ei,n n>=1 & 
   At,u (t>=i & 2*t<=2*i+3*n & u=t+n) => Q[t]=Q[u]":
\end{verbatim}

To handle condition (2), there are multiple possible approaches.   One is the following:  we compute a linear representation for the number of novel factors of length $n$ appearing in ${\bf q}$; these are factors ${\bf q}[i..i+n-1]$ that never appeared previously.   Then the number of novel factors is the factor complexity at $n$.  To do this, we use the following {\tt Walnut} code:
\begin{verbatim}
def factoreq "?msd_pisot4 Au,v (u>=i & u<i+n & u+j=v+i) => Q[u]=Q[v]":
def novel n "?msd_pisot4 n>=1 & Aj (j<n) => ~$factoreq(i,j,n)":
def twon n "?msd_pisot4 i<2*n":
\end{verbatim}
This would give us a linear representation for the factor complexity function $\rho_n({\bf q})$,
with the exception that it computes $0$ at $n=0$,
and a linear representation for $2n$.  We can then check whether these two linear representations compute the same function using
the ideas in \cite[p.~58]{Shallit:2023}.

However, our attempts to prove (1) and (2) using our $P4$ automaton failed because of the large number of states involved.
To get around this problem, we use the same approach, but with a different numeration system that is more ``tuned'' to the specific word $\bf q$.

Recently the first author and co-authors developed a new technique for creating addable numeration systems based on morphisms \cite{Carton&Couvreur&Delacourt&Ollinger:2024}.
Using this we can complete the computational proof of our main theorem, Theorem~\ref{thm2}.

\begin{proof}
The Dumont-Thomas numeration system associated with the morphism 
$${h(a) = ab, h(b) = cb, h(c) = a}$$ is given by the addressing automaton $\mathcal{N}_h$ in Fig.~\ref{fig:addh}. The language recognized by this DFA, ordered in radix order, is the abstract numeration system DT$_h$. We obtain a DFAO recognizing ${\bf q}$  by adding the output $x$ for each state $x$ of $\mathcal{N}_h$. By the Cayley-Hamilton theorem, all sequences $(|h^n(u)|)_{n\geq 0}$ obey the recurrence relation given by the characteristic polynomial $P(X) = X^3-2X^2+X-1$ of the incidence matrix $$M=\begin{pmatrix}1&1&0\\0&1&1\\1&0&0\end{pmatrix}$$ of the morphism $h$. It turns out $P$ is the irreducible monic polynomial of the Pisot number $$\theta = \frac{2 + \sqrt[3]{\frac{25+3\sqrt{69}}2} + \sqrt[3]{\frac{25-3\sqrt{69}}2}}3\approx 1.7548776662466916.$$ Without too much surprise, the number $\theta$ is the same Pisot number used to define the greedy numeration system $P4$. To compute the addition relation for DT$_h$, we construct the sequence automaton $\mathcal{A}_h$ given in Fig.~\ref{fig:seqh} by associating the sequence of lengths of the associated prefixes with each symbol on a transition of the addressing automaton as follows:
\begin{itemize}
    \item it is the null sequence $(0)$ for the symbol $0$ 
    \item for the symbol $1$ it is the sequence $(a_n) = (|h^n(a)|)_{n\geq 0}$ starting from state $a$ as $h(a)=ab$ and the sequence $(c_n) = (|h^n(c)|)_{n\geq 0}$ starting from state $b$ as $h(b)=cb$.
\end{itemize}
Given the recurrence relation ${u_{n+3} = 2u_{n+2} - u_{n+1} + u_n}$, it is sufficient to compute the first $3$ values of each sequence. We do so using the incidence matrix and the initial length 1 for each letter $a, b, c$.
\[
\begin{array}{cccc | cccc}
n & 0 & 1 & 2 & 3 & 4 & 5 & \cdots\\\hline
a_n & 1 & 2 & 4 & 7 & 12 & 21 & \cdots\\
b_n & 1 & 2 & 3 & 5 & 9 & 16 & \cdots\\
c_n & 1 & 1 & 2 & 4 & 7 & 12 & \cdots\\
\end{array}
\]
We recognize the sequence \seqnum{A005251} from the OEIS \cite{oeis} for both $(a_n)$ and $(c_n)$.

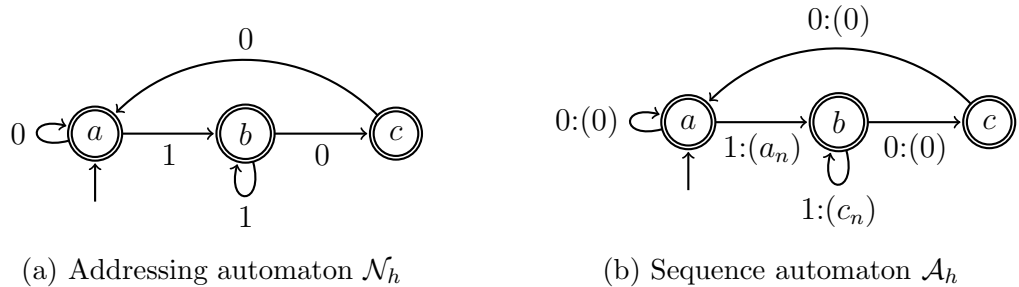
\begin{figure}[!ht]
\centering
\subcaptionbox{Addressing automaton $\mathcal{N}_h$\label{fig:addh}}[.45\textwidth]
{\begin{tikzpicture}[thick,shorten >=1pt,node distance=2cm,on grid,auto,initial text=] 
   \node[state,initial below,accepting] (q_a) {$a$}; 
   \node[state,accepting] (q_b) [right=of q_a] {$b$}; 
   \node[state,accepting] (q_c) [right=of q_b] {$c$}; 
    \path[->] 
    (q_a) edge [loop left] node {0} ()
          edge node[below] {1} (q_b)
    (q_b) edge [loop below] node  {1} ()
          edge node[below] {0} (q_c)
    (q_c) edge [bend right=45] node[above] {0} (q_a)
          ;
\end{tikzpicture}}
\subcaptionbox{Sequence automaton $\mathcal{A}_h$\label{fig:seqh}}[.45\textwidth]
{\begin{tikzpicture}[thick,shorten >=1pt,node distance=2cm,on grid,auto,initial text=] 
   \node[state,initial below,accepting] (q_a) {$a$}; 
   \node[state,accepting] (q_b) [right=of q_a] {$b$}; 
   \node[state,accepting] (q_c) [right=of q_b] {$c$}; 
    \path[->] 
    (q_a) edge [loop left] node {0:$(0)$} ()
          edge node[below] {1:$(a_n)$} (q_b)
    (q_b) edge [loop below] node  {1:$(c_n)$} ()
          edge node[below] {0:$(0)$} (q_c)
    (q_c) edge [bend right=45] node[above] {0:$(0)$} (q_a)
          ;
\end{tikzpicture}}
\caption{Addressing and sequence automata for ${h(a) = ab, h(b) = cb, h(c) = a}$}
\end{figure}

Combining techniques from Frougny and Solomyak \cite{frousol} to techniques from Bruyère and Hansel \cite{bruhan}, one can construct a DFA recognizing the addition relation $\{ (x,y,z) \suchthat x+y=z \}$. That is, the system DT$_h$ is addable, to use the terminology from \cite{Peltomaki&Salo}. For more details and a practical tool to compute this addition automaton, refer to \cite{Carton&Couvreur&Delacourt&Ollinger:2024}.

To construct a numeration system for ${\bf q} = g({\bf p})$, we first consider the length of images of $|g(a)| = |011| = 3$, $|g(b)| = |0| = 1$, and $|g(c)| = |01| = 2$ and use them as initial values to construct new sequences for an \emph{inflated} version ${\bf p}'$ of ${\bf p}$---its image by $g'(a)=a12$, $g'(b)=b$, $g'(c)=c3$. By construction ${\bf q} = g''({\bf p}')$ where $g''(a)= g''(b)= g''(c)=0$ and $g''(1)= g''(2)= g''(3)=1$.
\begin{align*}
{\bf p} &= abcbacbabacbabcbabacbabcbacbabcbabacbabcbacbabacbabcbacbabcbabacbabcba\cdots\\
{\bf p}' &= a12bc3ba12c3ba12ba12c3ba12bc3ba12ba12c3ba12bc3ba12c3ba12bc3ba12ba12\cdots
\end{align*}

\[
\begin{array}{cccc | cccc}
n & 0 & 1 & 2 & 3 & 4 & 5 & \cdots\\\hline
a'_n & 3 & 4 & 7 & 13 & 23 & 40 & \cdots\\
b'_n & 1 & 3 & 6 & 10 & 17 & 30 & \cdots\\
c'_n & 2 & 3 & 4 & 7 & 13 & 23 & \cdots\\
\end{array}
\]

We recognize the sequence \seqnum{A137495} from OEIS for both $(a'_n)$ and $(c'_n)$. Replacing the sequences $(a_n)$ and $(c_n)$ by $(a'_n)$ and $(c'_n)$ inside $\mathcal{A}_h$ provides a new sequence automaton for the positions with letters $a$, $b$ and $c$ in ${\bf p}'$. To obtain a numeration system for ${\bf p}'$, it is sufficient to provide new transitions to reach letters $1$, $2$ and $3$, which are all at bounded distance of a support letter (given by $g'$). To do so, we first replace the letter $1$ by the letter $3$ to preserve radix order, and we allow the last letter of a number representation to add an increment of $1$ or $2$ to its position and add a new state to represent the final symbol $1$ in ${\bf q}$. We then modify each transition pointing to $a$ or $c$ and replicate it via a transition to the new state. Doing so, we obtain the addressing automaton $\mathcal{N}_{\bf q}$ depicted in Fig.~\ref{fig:addq} that defines the new numeration system for ${\bf q}$. A DFAO for ${\bf q}$ is obtained by adding $g''(x)$ as output to each state $x$. Considering the new recurrence relation given by $X P(X) = X^4-2X^3+X^2-X$, we associate ultimately null sequences $\alpha_i = (i, 0, 0, 0, \ldots)$ with the increment transitions. This way, we obtain the sequence automaton $\mathcal{A}_{\bf q}$ depicted in Fig.~\ref{fig:seqq}, from which we compute the addition relation DFA for $\{ (x,y,z) \suchthat x+y=z \}$ using the same technique as before---it is still possible as the recurrence relation is ultimately Pisot.

\begin{figure}[!ht]
\centering
\subcaptionbox{Addressing automaton $\mathcal{N}_{\bf q}$\label{fig:addq}}[.45\textwidth]
{\begin{tikzpicture}[thick,shorten >=1pt,node distance=2cm,on grid,auto,initial text=] 
   \node[state,initial below,accepting] (q_a) {$a$}; 
   \node[state,accepting] (q_b) [right=of q_a] {$b$}; 
   \node[state,accepting] (q_c) [right=of q_b] {$c$};
   \node[state,accepting] (q_1) [below=of q_b] {$1$};
    \path[->] 
    (q_a) edge [loop left] node {0} ()
          edge node[below] {3} (q_b)
          edge [bend right=20] node[below left] {1,2} (q_1)
    (q_b) edge [loop above] node  {3} ()
          edge node[below] {0} (q_c)
          edge node {1} (q_1)
    (q_c) edge [bend right=90] node[above] {0} (q_a)
          edge [bend left=20] node[below right] {1,2} (q_1)
          ;
\end{tikzpicture}}
\subcaptionbox{Sequence automaton $\mathcal{A}_{\bf q}$\label{fig:seqq}}[.45\textwidth]
{\begin{tikzpicture}[thick,shorten >=1pt,node distance=2cm,on grid,auto,initial text=] 
   \node[state,initial below,accepting] (q_a) {$a$}; 
   \node[state,accepting] (q_b) [right=of q_a] {$b$}; 
   \node[state,accepting] (q_c) [right=of q_b] {$c$};
   \node[state,accepting] (q_1) [below=of q_b] {$1$};
    \path[->] 
    (q_a) edge [loop left] node {0:$(0)$} ()
          edge node[below] {3:$(a'_n)$} (q_b)
          edge [bend right=20] node[below left] {1:$\alpha_1$,2:$\alpha_2$} (q_1)
    (q_b) edge [loop above] node  {3:$(c'_n)$} ()
          edge node[below] {0:$(0)$} (q_c)
          edge node {1:$\alpha_1$} (q_1)
    (q_c) edge [bend right=90] node[above] {0:$(0)$} (q_a)
          edge [bend left=20] node[below right] {1:$\alpha_1$,2:$\alpha_2$} (q_1)
          ;
\end{tikzpicture}}
\caption{Addressing and sequence automata for ${\bf q}$}
\end{figure}

This gives a family of relatively small automata: the DFA for the numeration system and the DFAO for ${\bf q}$ both have 4 states, and the addition relation DFA has 143 states. With this new representation, manipulating ${\bf q}$ with Walnut is very fast and we conclude the desired results, as follows:  
\begin{verbatim}
eval check52plus "?msd_mor ~Ei,n n>=1 & At,u (t>=i & 2*t<=2*i+3*n & u=t+n) 
   => Mor[t]=Mor[u]":
def factoreq "?msd_mor Au,v (u>=i & u<i+n & u+j=v+i) => Mor[u]=Mor[v]":
# 125 states
def novel n "?msd_mor n>=1 & Aj (j<n) => ~$factoreq(i,j,n)":
# 120 states
def twon n "?msd_mor n>=1 & i<2*n":
# 70 states
\end{verbatim}
and we now check, using a Maple program,
that the linear representations computed by
{\tt novel} and {\tt twon} are identical.
Here {\tt Mor} is the automaton computing {\tt Q} in the new numeration system.
\end{proof}

The attentive reader might notice that the addressing automaton $\mathcal{N}_{\bf q}$ is prefix-closed, as all states are finite. Indeed, it is related (by exchanging $3$ and $2$ in the label from $b$ to $b$) to the Dumont-Thomas numeration system associated with the morphism $t(a)=a11b$, $t(1)=\varepsilon$, $t(b)=c1b$, $t(c)=a11$.

\section{Additional remarks}

\begin{remark}
We observe that $\bf q$ is uniformly recurrent.   Indeed, two factors of length $n\geq 1$ are separated by at most $7n$
positions.
\begin{verbatim}
def nextgap "?msd_mor Ej i<j & $factoreq(i,j,n) & i+g=j & 
   At (i<t & t<j) => ~$factoreq(i,t,n)":
# 275 states
def maxgap "?msd_mor Ei $nextgap(g,i,n) & 
   Ah (h>g) => ~Ei $nextgap(h,i,n)"::
# 38 states
eval uc "?msd_mor Ag,n (n>=1 & $maxgap(g,n)) => g<=7*n"::
\end{verbatim}
And {\tt Walnut} returns {\tt TRUE} for this last assertion.
\end{remark}

\begin{remark}
As it turns out, the factor $1001\,1001\,10$ mentioned previously is the {\it only\/} factor of exponent $5/2$ in $\bf q$.   We can prove this
with the following {\tt Walnut} commands:
\begin{verbatim}
def per "?msd_mor p>=1 & p<=n & $factoreq(i,i+p,n-p)":
# 719 states
def exp52 "?msd_mor Ep $per(i,n,p) & 2*n=5*p":
# 7 states
eval testlength "?msd_mor Ai,n (n>=1 & $exp52(i,n)) => 
   (n=10 & $factoreq(i,11,10))":
\end{verbatim}
This raises the question of what the {\it asymptotic critical exponent} of $\bf q$ is:   this is the supremum over all $e$ such that there are infinitely many distinct factors of exponent $\leq e$.   It seems that
this asymptotic critical exponent is the same as that for $\bf p$, but we did not try to prove this, although it seems likely it can be done using {\tt Walnut}.
\end{remark}

\begin{remark}
The sequence $\bf q$ contains no palindromes of length $>15$.   Indeed, it even contains no reversible factor of length $>15$, where ``reversible" means a factor whose reverse is also in $\bf q$.  This can be proven most simply by checking all $32$ factors of length $16$.
\end{remark}

\begin{remark}
The set of all finite binary Rote words of critical exponent $\leq 5/2$ seems to have a certain ``rigidity'' of structure, reminiscent of the famous Restivo-Salemi theorem for overlap-free words \cite{Restivo&Salemi:1985a}.  For example,
it seems like the number of such words of length $n$ is $\leq 16n$ for all
$n \geq 58$.  However, we leave this for another day.
\end{remark}

%  Other things to mention and/or think about
% --  q  has an infinite string attractor  

\end{document}